\documentclass[11pt,reqno,oneside]{amsart}
\usepackage[utf8]{inputenc}
\usepackage[T2A]{fontenc}
\usepackage[english]{babel}
\usepackage{amsmath, amssymb, amsthm}
\usepackage[pdfborder={0 0 0}, pdffitwindow = true, pdfstartview = FitH,  pdftitle = {On the minimum number of facets of a 2-neighborly polytope},
pdfauthor = {A.N. Maksimenko}
]{hyperref}

\usepackage[margin=2.75cm]{geometry}

\usepackage{subcaption}

\usepackage{diagbox} 
\usepackage{booktabs} 

\usepackage{multirow}
\usepackage{rotating}

\theoremstyle{plain}
\newtheorem{theorem}{Theorem}
\newtheorem{lemma}[theorem]{Lemma}

\theoremstyle{definition}

\theoremstyle{remark}

\newcommand{\R}{\mathbb{R}}    % множество действительных чисел
\newcommand{\N}{\mathbb{N}}    % множество натуральных чисел
\newcommand{\No}{\textnumero}    % russian symbol of number
\DeclareMathOperator{\conv}{conv}
\DeclareMathOperator{\vertices}{vert}

\newcommand{\mn}{\mu_{\text{2n}}}
\newcommand{\msn}{\mu^{\text{s}}_{\text{2n}}}

\usepackage{mathtools}
\providecommand\given{} % The vertical rule for conditions
\newcommand\SetSymbol[1][]{%
	\nonscript\:#1\vert
	\allowbreak \nonscript\:	\mathopen{}}
\DeclarePairedDelimiterX\Set[1]\{\}{%
	\renewcommand\given{\SetSymbol[\delimsize]}	#1}

\title{On the minimum number of facets \\of a 2-neighborly polytope}

\author{Aleksandr Maksimenko}
\thanks{Supported by the~project \No\,1.5768.2017/9.10 of P.\,G.~Demidov Yaroslavl State University within State Assignment for~Research.}
\address{Laboratory of Discrete and Computational Geometry, P.G. Demidov Yaroslavl State University, ul. Sovetskaya 14, Yaroslavl 150000, Russia} 
\email{maximenko.a.n@gmail.com}
\begin{document}

\begin{abstract}
Let $\mn(d,v)$ (respectively, $\msn(d,v)$) be the minimal number of facets of a (simplicial) 2-neigh\-bor\-ly $d$-po\-ly\-tope with $v$ vertices, $v > d \ge 4$.
It is known that $\mn(4,v) = v (v-3)/2$, $\mn(d, d+2) = d+5$, $\mn(d,d+3) = d+7$ for $d \ge 5$, and $\mn(d,d+4) \in [d+5, d+8]$ for $d \ge 6$.
We show that $\mn(5, v) = \Omega(v^{4/3})$, $\mn(6, v) \ge v$,
and the equality $\mn(6, v) = v$ holds only for a simplex and for a dual 2-neighborly 6-polytope (if it exists) with $v \ge 27$.
By using $g$-theorem, we get $\msn(d, v) = \Delta (\Delta(d-3) + 3d - 5)/2 + d + 1$, where $\Delta = v - d - 1$.
Also we show that $\mn(d, v) \ge d+7$ for $v \ge d+4$.
\end{abstract}

\maketitle

\section{Introduction}

Let $P$ be a $d$-polytope, i.e., a $d$-dimensional convex polytope.
An $i$-dimensional face of $P$ is called \emph{$i$-face},
0-faces are \emph{vertices}, 1-faces are \emph{edges}, $(d-1)$-faces are \emph{facets},
 and $(d-2)$-faces are \emph{ridges}.
A polytope $P$ is \emph{simplicial} if every facet is a simplex, and $P$ is \emph{simple} if it is dual to a simplicial polytope.

Let $f_i(P)$ be the number of $i$-faces of $P$, $0 \le i \le d-1$.
The problem of estimating $f_i(P)$ (where $P$ belongs to some class of polytopes) in terms of $f_0(P)$ is well known.
For the class of simplicial polytopes, the problem is known as the upper bound and the lower bound theorems
(see~\cite[Chap.~10]{Grunbaum:2003} for details).
In~particular~\cite{Barnette:1971},
\begin{equation}
\label{eq:LowBoundSimplex}
f_{d-1} (P) \ge (d-1)(f_0 (P)-d) + 2 \quad \text{for a simplicial $d$-polytope $P$.}
\end{equation}
In~1990, G.~Blind and R.~Blind~\cite{Blind:1990} solved the lower bound problem for the class of polytopes without a triangle 2-face.
We raise the question for the class of 2-neigh\-bor\-ly polytopes.

A $d$-polytope $P$ is called \emph{$k$-neigh\-bor\-ly} if every subset of $k$ vertices forms the vertex set of some face of $P$.
Since every $d$-polytope is 1-neigh\-bor\-ly, we will consider
 $k$-neigh\-bor\-ly polytopes only for nontrivial cases $k \ge 2$.
For $k > d/2$, there is only one combinatorial type of a $k$-neigh\-bor\-ly $d$-polytope~--- a $d$-simplex~\cite[p.~123]{Grunbaum:2003}.
The same is true for $f_0(P) = d+1$.
Therefore, we suppose $d \ge 2k$ and $f_0(P) > d+1$.

A $\lfloor d/2\rfloor$-neigh\-bor\-ly polytope is called \emph{neigh\-bor\-ly}.
In~particular, every neigh\-bor\-ly $d$-polytope is 2-neigh\-bor\-ly for $d \ge 4$.
The family of neighborly polytopes is constantly attracting the attention of researchers. % (see, e.g., \cite{Grunbaum:2003}).
For $d = 2k$, these polytopes have the maximal number of facets over all $d$-polytopes with $n$ vertices~\cite{McMullen:1970}:
\begin{equation}
\label{EqNeighborly}
  f_{d-1}(P_{\text{neighborly}}) = \frac{n}{n-k} \binom{n - k}{k}.
\end{equation}

There exists a widespread feeling that $k$-neigh\-bor\-ly polytopes are very common among convex polytopes~\cite[p.~129--129a]{Grunbaum:2003}, \cite[sec.~3.4]{Gillmann:2006}, \cite{Padrol:2013}.
Moreover, combinatorial polytopes of NP-hard problems has $k$-neighborly faces with superpolynomial number of vertices~\cite{Maksimenko:2014, Maksimenko:2015}. 

%As a reference point for further investigations we pose the following conjecture.

%\begin{conjecture}
%$f_{d-1}(P) \ge f_{0}(P)$ for a $2$-neigh\-bor\-ly $d$-polytope $P$.
%The number of facets $f_{d-1}(P)$ of a $2$-neigh\-bor\-ly $d$-polytope $P$  cann't be less than the number of its vertices $f_0(P)$.
%\end{conjecture}

%From \eqref{LowBoundSimplex}, it follows that the conjecture is true if $P$ is simplicial.
%The case $d = 4$ is covered by \eqref{EqNeighborly}.

We will denote by $\mn(d,v)$ the minimal number of facets of a 2-neigh\-bor\-ly $d$-po\-ly\-tope with $v$ vertices. From~\eqref{EqNeighborly}, we get
$\mn(4,v) = v(v-3)/2$.
In~\cite{Maksimenko:2010}, there was posed the following question: $\mn(d,v) \ge v$? This inequality was validated for two cases~\cite{Maksimenko:2010}: $d \le 6$ and $v \le d+5$.

In~Section~\ref{sec:examples}, we list examples of 2-neighborly $d$-polytopes with small difference $f_{d-1}(P) - f_0(P)$.
%\[
% f_{d-1}(P) - f_0(P) \approx \frac{f_0(P) (f_0(P) - d - 1)}{0.4 d}.
%\]

In~Section~\ref{sec:smalldim}, for 5- and 6-polytopes we show that
\[
\mn(5, v) \ge \min_{n \in \N} \max \left\{\frac{v (v - 1)}{n},  \frac{n(n - 3)}{2} + 1\right\} = \Omega(v^{4/3})
\]
and $\mn(6, v) \ge v$ and the equality $\mn(6, v) = v$ holds only for a simplex and for a dual 2-neighborly polytope with $v \ge 27$.

In~Section~\ref{sec:smallvert}, by using $g$-theorem, we get the lower bound for the number of $i$-faces of a simplicial 2-neigh\-bor\-ly $d$-polytope. 
Let $\msn(d,v)$ be the minimal number of facets of a simplicial 2-neigh\-bor\-ly $d$-po\-ly\-tope with $v$ vertices.
Obviously, $\mn(d, v) \le \msn(d, v)$.
We prove that
\begin{equation} 
\label{eq:LBs2n}
\msn(d, v) = \frac{\Delta (\Delta(d-3) + 3d - 5)}{2} + d + 1, \quad \text{where } \Delta = v - d - 1.
\end{equation} 
Also we show that $\mn(d, v) \ge d+7$ for $v \ge d+4$.

%%%%%%%%%%%%%%%%%%%%%%%%%%%%%%%%%%%%%%%%%%%%%%%%%%%%%%%%%%%%%%%%%%%%%%%%%
% 
% Examples
% 
%%%%%%%%%%%%%%%%%%%%%%%%%%%%%%%%%%%%%%%%%%%%%%%%%%%%%%%%%%%%%%%%%%%%%%%%%

\section{Examples of 2-neighborly polytopes with small number of facets}
\label{sec:examples}

In this section we list all known (to the author) from the literature examples of 2-neighborly polytopes with small number of facets.

First of all, it is well known that $f_3 (P) = f_0(P) (f_0(P) - 3) / 2$ for a 2-neighborly 4-polytope $P$. In particular, this is true for cyclic polytopes.

Let $P$ be a 2-neighborly $d$-polytope. If $Q$ is an $r$-fold pyramid with basis $P$, then $Q$ is a~2-neigh\-bor\-ly $(d+r)$-polytope with $f_0(Q) = f_0(P) + r$ and $f_{d+r-1}(Q) = f_{d-1}(P) + r$~\cite[Sec.~4.2]{Grunbaum:2003}. Hence, by using cyclic 4-polytopes, for every $d > 4$ and $v > d$ we can construct an example of a~2-neighborly $d$-polytope $Q$ with $v$ vertices and $(v+4-d)(v+1-d)/2 + d-4$ facets, i.e.
\begin{equation}
\label{eq:d4}
f_{d-1}(Q) - f_0(Q) = (f_0(Q)+4-d)(f_0(Q)-d-1)/2.
\end{equation}

It is also well known that every $d$-polytope $P$ with $d+2$ vertices is a simplicial one (and inequality \eqref{eq:LowBoundSimplex} holds for it) or a pyramid over a $(d-1)$-polytope with $d+1$ vertices. Consequently, if it is 2-neighborly, then $f_{d-1}(P) \ge d+5$ and the equality is attained only on a $(d-4)$-fold pyramid over a cyclic 4-polytope with 6 vertices. It is interesting to remark that the last one can be realized as a 0/1-polytope (see fig.~\ref{fig:P46}).

\begin{figure}
	\begin{minipage}[t]{.21\linewidth}
		\centering
		(0, 0, 0, 0),\\
		(0, 0, 0, 1),\\
		(0, 0, 1, 0),\\
		(0, 1, 0, 0),\\
		(1, 0, 0, 1),\\
		(1, 1, 1, 0).
		\subcaption{$P_{4,6}$, $f=9$}\label{fig:P46}
	\end{minipage}%
	\begin{minipage}[t]{.24\linewidth}
		\centering
		(0, 0, 0, 0, 0),\\
		(0, 0, 0, 0, 1),\\
		(0, 0, 0, 1, 0),\\
		(0, 0, 1, 0, 0),\\
		(0, 1, 0, 0, 1),\\
		(0, 1, 1, 1, 0),\\
		(1, 0, 0, 1, 0),\\
		(1, 0, 1, 0, 1).
		\subcaption{$P_{5,8}$, $f=12$}\label{fig:P58}
	\end{minipage}
	\begin{minipage}[t]{.24\linewidth}
		\centering
		(0, 0, 0, 0, 0),\\
		(0, 0, 0, 0, 1),\\
		(0, 0, 0, 1, 0),\\
		(0, 0, 1, 0, 1),\\
		(0, 1, 0, 1, 0),\\
		(0, 1, 1, 0, 0),\\
		(1, 0, 0, 1, 1),\\
		(1, 0, 1, 0, 0),\\
		(1, 1, 0, 0, 0).
		\subcaption{$P_{5,9}$, $f=16$}\label{fig:P59}
	\end{minipage}
	\begin{minipage}[t]{.29\linewidth}
		\centering
		(0, 0, 0, 0, 0, 0),\\
		(0, 0, 0, 0, 0, 1),\\
		(0, 0, 0, 0, 1, 0),\\
		(0, 0, 0, 1, 0, 0),\\
		(0, 0, 1, 0, 0, 0),\\
		(0, 1, 0, 0, 0, 0),\\
		(1, 0, 0, 0, 1, 1),\\
		(1, 0, 1, 1, 0, 0),\\
		(1, 1, 0, 1, 0, 1),\\
		(1, 1, 1, 0, 1, 0).
		\subcaption{$P_{6,10}$, $f=14$}\label{fig:P610}
	\end{minipage}
	\caption{Examples of $d$-dimensional 2-neighborly 0/1-polytopes $P_{d,v}$ with the minimal number of facets $f$ for the given number of vertices $v$.}\label{fig:01}
\end{figure}

In~\cite{Fukuda:2013}, there are enumerated all combinatorial types of 5-polytopes with $\le 9$ vertices (all face lattices available at \url{http://www-imai.is.s.u-tokyo.ac.jp/~hmiyata/oriented_matroids/}). For 2-neighborly 5-polytopes with 8 vertices, the minimal number of facets is 12 and it is attained on the 0/1-polytope on the fig.~\ref{fig:P58}.
The minimal number of facets of a 2-neighborly 5-polytope with 9 vertices is equal to 16 and the polytope also can be realized as a 0/1-polytope (see fig.~\ref{fig:P59}).

By enumerating minimal reduced Gale diagrams for 2-neighborly $d$-polytopes with $d+3$ vertices, one can conclude~\cite{Maksimenko:2018} that $\mn(d,d+3) = d+7$ %the minimum for the difference $f_{d-1}(P) - f_{0}(P)$ is equal to 4
and the lower bound is attained on $P_{5,8}$ (see fig.~\ref{fig:P58}).

In~\cite{Aichholzer:2000}, O.~Aichholzer enumerated all 2-neighborly 0/1-polytopes of dimension~$\le 6$. For dimension 5, except the examples listed in fig.~\ref{fig:01}, there is a polytope with 10 vertices and 22 facets. For dimension 6, if the number of vertices is 10, then the minimum number of facets for such polytopes is equal to 14~(see fig.~\ref{fig:P610}). For the other numbers of vertices 11, 12, and 13, minimum numbers of facets equal 17, 21, and 26.

All the mentioned results are summarized in Table~\ref{tab:1}.
\begin{table}
	\begin{tabular}{|c|c|c|c|c|c}
		%\toprule
		\hline
		\diagbox{$d$}{$v$} & $d+2$ & $d+3$ & $d+4$ & $d+5$ & \dots \\ 
		%\midrule
		\hline
		%\phantom{aaaa}
		4 & \multirow{4}[b 7]*{\begin{sideways}$d+5$\end{sideways}} &    \multicolumn{4}{c}{$v(v-3)/2$} \vphantom{$\dfrac12$}\\ 
		\cline{1-1}  \cline{3-6}  
		5 &  & \multirow{3}[b 4]*{\begin{sideways}$d+7$\end{sideways}} & 16 & $\le 22$ & \vphantom{$\dfrac12$}\\ 
		\cline{1-1}  \cline{4-5}  
		6 &  &  & $\le 14$ & $\le 17$ & \vphantom{$\dfrac12$}\\ 
		\cline{1-1}  \cline{4-5}  
		$\vdots$ & & & \multicolumn{3}{c}{} \\
		%\bottomrule
%		\hline
	\end{tabular} 
	\\[1ex]
	\caption{The minimal number of facets of a 2-neigh\-bor\-ly $d$-po\-ly\-tope with $v$ vertices.}
	\label{tab:1}
\end{table}

It is natural to try to construct examples of a 2-neighborly $d$-polytopes with as small as possible difference between facets and vertices. Some progress can be done with a \emph{join} of two polytopes~\cite{Henk:2004}:
\[
P * P' := \conv\left(\Set*{(x, 0, 0) \in \R^{d+d'+1} \given x \in P} \cup 
                     \Set*{(0, y, 1) \in \R^{d+d'+1} \given y \in P'}\right),
\]
where $P$ is a $d$-polytope and $P'$ is a $d'$-polytope.
The polytope $P * P'$ has dimension $d+d'+1$, $f_0(P) + f_0(P')$ vertices, and $f_{d-1}(P) + f_{d'-1}(P')$ facets~\cite[Sec.~15.1.3]{Henk:2004}. 
Moreover, if $P$ and $P'$ are $k$-neighborly, then $P * P'$ is also $k$-neighborly.

Let $P_n^0$ be a 2-neighborly 4-polytope with $n$ vertices, $n \ge 5$, and let $P_n^1 = P_n^0 * P_n^0$.
Hence $P_n^1$ is a 2-neighborly 9-polytope with $2n$ vertices and $n(n-3)$ facets.
Let us define $P_n^m$ recursively:
\[
  P_n^{m+1} = P_n^m * P_n^m, \quad m \in \N.
\]
Thus $P_n^m$ is a 2-neighborly $d$-polytope with
\[
 d = 5 \cdot 2^m - 1, \quad f_0(P_n^m) = 2^m n, \quad f_{d-1}(P_n^m) = 2^{m-1} n (n-3).
\]
Therefore,
\[
 f_{d-1}(P_n^m) - f_0(P_n^m) = \frac{f_0(P_n^m) (f_0(P_n^m) - d - 1)}{2^{m+1}} < \frac{f_0(P_n^m) (f_0(P_n^m) - d - 1)}{0.4 d}.
\]
This difference has a bit better assymptotic than~\eqref{eq:d4}.

%%%%%%%%%%%%%%%%%%%%%%%%%%%%%%%%%%%%%%%%%%%%%%%%%%%%%%%%%%%%%%%%%%%%%%%%%
% 
% Small dimension
% 
%%%%%%%%%%%%%%%%%%%%%%%%%%%%%%%%%%%%%%%%%%%%%%%%%%%%%%%%%%%%%%%%%%%%%%%%%

\section{Small dimensions}
\label{sec:smalldim}

A polytope $P$ is called \emph{$m$-simplicial} if every $m$-face of $P$ is a simplex. A polytope is called \emph{$m$-simple} if it is dual to an $m$-simplicial polytope.
In this section we use the well known fact that a $k$-neighborly polytope is $(2k - 1)$-simplicial~\cite[Sec.~7.1]{Grunbaum:2003}.
The set of all $i$-faces of a $d$-polytope $P$ we denote by $F_i(P)$, $i = 0, 1, \dots, d-1$.
%A polytope is \emph{$m$-simple} if it is dual to $m$-simplicial. 

The two statements below (Lemma~\ref{lem:SimplexBound} and Theorem~\ref{thm:3k-1}) are generalized versions of the results from~\cite{Maksimenko:2010}.

\begin{lemma}
\label{lem:SimplexBound}
Let $P$ be an $m$-simplicial $d$-polytope and $m \ge d/2$.
Then $f_{m}(P) \ge f_{d-m-1}(P)$ and the equality is attained only
if $P$ is $m$-simple.
\end{lemma}

\begin{proof}
Let us count incidences between $i$-faces and $(i-1)$-faces of $P$, $0 < i \le m$.
Note that every $(i-1)$-face of a $d$-polytope is incident with at least $(d-i+1)$ $i$-faces.
Hence,
\[
(d-i+1) f_{i-1} (P) \le \hspace{-0.5em} \sum_{p \in F_i(P)} \hspace{-0.5em} f_{i-1}(p), %\{\text{$i$-faces of $P$}\}
\]	
 where %$F_i(P)$ is the set of all $i$-faces of~$P$,
$f_{i-1}(p)$ is the number of $(i-1)$-faces of~$p$.
Since $i$-faces are simplices,
\begin{equation}
	\label{eq:Incidence}
  (d-i+1) f_{i-1}(P) \le (i+1) f_i(P) \qquad \text{for } 0 < i \le m.
\end{equation}
Note that the equality here is fulfilled if and only if every $(i-1)$-face of $P$ is incident to exactly $(d-i+1)$ $i$-faces, i.e. $P$ is $(d-i)$-simple.

Let $d$ be even, $d = 2n$, $n\in\N$. Let $i = n$. 
From inequality \eqref{eq:Incidence}, we get
\begin{equation}
	\label{eq:IncidFirst}
 (n+1) f_{n-1}(P) \le (n+1) f_n(P).
\end{equation}
Suppose that $m > n$. 
Thus, substituting $i \in \{n-1, n+1\}$ in \eqref{eq:Incidence}, we obtain
\[
 (n+2) f_{n-2}(P) \le n f_{n-1}(P) \quad \text{and} \quad
 n f_{n}(P) \le (n+2) f_{n+1}(P).
\]
Combining this with \eqref{eq:IncidFirst}, we have
\[
 f_{n-2}(P) \le f_{n+1}(P).
\]
 By repeating this procedure, it is easy to get
\[
 f_{d-m-1}(P) \le f_{m}(P) \quad \text{for even $d$.}
\]
Moreover, the equality is attained only on $m$-simple polytopes.

The case $d = 2n+1$, $n\in\N$, are proved by analogy. (By using the duality and the fact that a $d$-simplex has the minimal number of faces among all $d$-polytopes~\cite[p.~36, Ex.~8]{Grunbaum:2003}.)
\end{proof}

\begin{theorem}
\label{thm:3k-1}
Let $P$ be a $k$-neighborly $d$-polytope and $2k < d \le 3k - 1$.
Then $f_{d-1}(P) \ge f_{0}(P)$ and the equality is attained only if $P$ is a simplex.
\end{theorem}

\begin{proof}
Recall that every $i$-face of a $k$-neigh\-bor\-ly $d$-polytope $P$ is a simplex for $i < 2k$~\cite[Sec.~7.1]{Grunbaum:2003}.
Using Lemma~\ref{lem:SimplexBound}, we get
\begin{equation}
	\label{eq:SimplexConcl}
 f_{d-k}(P) \ge f_{k-1}(P).
\end{equation}

Note also that for a $k$-neigh\-bor\-ly $d$-polytope $P$ we can use the implication
\begin{equation}
	\label{eq:Implication}
  f_{k-1}(P) \le f_{d-k}(P) \Rightarrow f_0(P) \le f_{d-1}(P).
\end{equation}
Indeed,  $f_{k-1}(P) = \binom{f_0(P)}{k}$ for a $k$-neighborly polytope $P$. But $f_{d-k}(P) \le \binom{f_{d-1}(P)}{k}$ for any $d$-polytope $P$ and the equality is attained only if $P$ is dual $k$-neighborly.
Finally note that 
\[
\binom{f_0(P)}{k} \le \binom{f_{d-1}(P)}{k} \Rightarrow f_0(P) \le f_{d-1}(P).
\] 

Combining \eqref{eq:SimplexConcl} and \eqref{eq:Implication}, we obtain
\[
  f_0(P) \le f_{d-1}(P) \quad \text{for $d \le 3k - 1$}.
\]	
Moreover, the equality is attained only if $P$ is dual $k$-neighborly. Hence, it is $(2k-1)$-simplicial and $(2k-1)$-simple, and must be a simplex, since $2k-1 + 2k-1 > d$~\cite[Exercise~4.8.11]{Grunbaum:2003}.
\end{proof}

\subsection{Dimension 5}

\begin{lemma}
\label{lem:2n5}
Let $P$ be a $2$-neighborly $5$-polytope.
If $v \in \vertices P$ and $f_P(v)$ is the number of facets incident with $v$, then $f_P(v) \ge |\vertices P| - 1$.
The equality $f_P(v) = |\vertices P| - 1$ is valid for all $v \in \vertices P$ if and only if $P$ is 2-simple.
\end{lemma}

\begin{proof}
Let $P$ be a $2$-neighborly $5$-polytope and $v \in \vertices P$.
Let $Q$ be a \emph{vertex figure of $P$ at $v$}, 
i.e. $Q$ is the intersection of $P$ by a hyperplane which strictly
separates $v$ from $\vertices P \setminus \{v\}$~\cite{Grunbaum:2003}.
Thus, $Q$ is a $4$-polytope and $f_P(v) = f_3(Q)$.
On the other hand, $f_0(Q) = |\vertices P| - 1$, since $P$ is 2-neighborly.

All $3$-faces of $P$ are simplices. Hence, all $2$-faces of $Q$ are triangles.
Note that every $1$-face of a $4$-polytope is incident with at least three $2$-faces, and if there are exactly three $2$-faces for every $1$-face, then the polytope is $2$-simple.
By counting incidences between $1$-faces and $2$-faces of $Q$, we have 
\[
 f_2(Q) * 3 \ge f_1(Q) * 3,
\]
and the equality is attained only if $Q$ is $2$-simple.
From Euler's equation~\cite[Sec.~8.1]{Grunbaum:2003}, we get
\[
 f_3(Q) - f_0(Q) = f_2(Q) - f_1(Q) \ge 0.
\]
\end{proof}

Now we want to obtain some lower bound for $f_4(P)$.
By counting incidences between vertices and facets of a 2-neighborly 5-polytope $P$ and using Lemma~\ref{lem:2n5}, we get:
\[
	\sum_{p \in F_4(P)} \hspace{-0.5em} f_0(p) \ge f_0(P) (f_0(P) - 1).
	%\text{$y$ is a facet of $P$}   \hspace{-2em}
\]
Let $\hat{x}$ be the average number of vertices in a facet of $P$:
\[
	\hat{x} = \frac{1}{f_4(P)} \sum_{p \in F_4(P)} \hspace{-0.5em} f_0(p).
\]
Hence,
\begin{equation}
\label{eq:f_4(P)}
f_4(P) \ge \frac{f_0(P) (f_0(P) - 1)}{\hat{x}}.
\end{equation}
%(So, if we fix $\hat{x}$, then $f_4 = \bigO(f_0^2)$.)
On the other hand, the number of facets of $P$ must be greater than 
the number of facets of a facet of $P$. 
Since every facet of $P$ is a $2$-neighborly $4$-polytope, 
\[
f_4(P) \ge \frac{\lceil \hat{x}\rceil(\lceil \hat{x}\rceil - 3)}{2} + 1.
\]
Hence,
\[
f_4(P) \ge \max \left\{\frac{f_0(P) (f_0(P) - 1)}{\hat{x}},  \frac{\lceil \hat{x}\rceil(\lceil \hat{x}\rceil - 3)}{2} + 1\right\}.
\]
Therefore, for large $f_0$ we have
\[
f_4 = \Omega(f_0^{4/3}).
\] 

%For finding lower bound, we must solve equation
%\[
%\frac{f_0(P) (f_0(P) - 1)}{x} =  \frac{x(x-3)}{2} + 1.
%\]
%Or, the same,
%\[
%x^3 - 3x^2 + 2x = 2 f_0 (f_0 - 1). 
%\]
%Since $x \ge 5$,
%\[
%	x \ge 1 + \sqrt[3]{2 f_0 (f_0 - 1) + 4}
%\]
%and
%\[
%f_4 \ge \frac{x(x-3)+2}{2} \ge \frac{S(f_0) \left( S(f_0) - 1\right)}{2},
%\]
%where $S(n) = \sqrt[3]{2 n (n - 1) + 4} \ge 4$ for $n \ge 6$.
%It is easy to check that the right hand of the inequality is greater than $f_0 + 1$ for $f_0 > 6$.

Another question raised by Lemma~\ref{lem:2n5} is as follows. Is there a 2-neighborly 2-simple 5-polytope other than a simplex? Obviously, every vertex figure of such polytope is a 2-simplicial 2-simple 4-polytope. Recently, Brinkmann and Ziegler~\cite{Brinkmann:2017} established that there are only seven 2s2s 4-po\-ly\-to\-pes with at most 12 vertices. By using this list we can proof the following.

\begin{theorem}
\label{thm:5d}
If $P$ is a 2-neighborly 2-simple 5-polytope other than a simplex,
then $f_0(P) \ge 14$.
\end{theorem}

\begin{proof}
Let $Q$ be a vertex figure of $P$ at $v$, $v \in \vertices P$. Consequently, $Q$ is a 2-simple 2-simplicial 4-polytope with $f_0(Q) = f_0(P) - 1$. %From~\cite[Theorem~7.2.13]{Werner:2009}, we know that $f_0(Q) \ge 9$.
Below we use the list of all such polytopes with $f_0(Q) \le 12$ from~\cite[Theorem~2.1]{Brinkmann:2017}.

If $f_0(Q) = 9$, then $f_1(Q) = 26$. Hence, every vertex in $P$ is incident to 26 2-faces of $P$. On the other hand, every 2-face of $P$ is a triangle. Therefore, $26 f_0(P)$ must be equal to $3 f_2(P)$. But $f_0(P) = 10$ and $f_2(P)$ is not integer.

If $f_0(Q) = 10$, then $f_2(Q) = 30$. Hence, every vertex in $P$ is incident to 30 3-faces of $P$. On the other hand, every 3-face of $P$ is a simplex. Therefore, $30 f_0(P)$ must be equal to $4 f_3(P)$. But $f_0(P) = 11$ and $f_3(P)$ is not integer.

If $f_0(Q) = 11$, then $Q$ has one facet with 7 vertices. Hence, every vertex in $P$ is incident to one facet with 8 vertices. But $P$ has 12 vertices.

If $f_0(Q) = 12$, then $f_2(Q) = 39$. Hence, every vertex in $P$ is incident to 39 3-faces of $P$. On the other hand, every 3-face of $P$ is a simplex. Therefore, $39 f_0(P)$ must be equal to $4 f_3(P)$. But $f_0(P) = 13$ and $f_3(P)$ is not integer.
\end{proof}

\subsection{Dimension 6}

In~\cite{Maksimenko:2010}, it was showed that $f_{5}(P) \ge f_{0}(P)$ for a 2-neighborly 6-polytope. Now we can update this result.

\begin{theorem}
	Let $P$ be a $2$-neighborly $6$-polytope. Then $f_{5}(P) \ge f_{0}(P)$.
	The equality $f_{5}(P) = f_{0}(P)$ holds if and only if $P$ is a simplex or $P$ is a dual 2-neighborly polytope with $f_{0}(P) \ge 27$.
\end{theorem}

\begin{proof}
Let $f_i := f_i(P)$.
From Euler's equation~\cite{Grunbaum:2003}, we get
\begin{equation*}
	f_0 - f_1 + f_4 - f_5 = f_3 - f_2.
\end{equation*}
Using Lemma~\ref{lem:SimplexBound}, we have $f_3 - f_2 \ge 0$ 
and $f_1 - f_0 \le f_4 - f_5$.
Note also, that $f_1 = f_0 (f_0-1)/2$ and $f_4 \le f_5 (f_5-1)/2$, and the equality is attained only for dual 2-neighborly polytopes.
Therefore, \(f_0 (f_0-3)/2 \le f_5 (f_5-3)/2\) and $f_0 \le f_5$.

Now, let 2-neighborly 6-polytope $P$ be dual 2-neighborly and $P$ is not a simplex. Let $T$ be a facet of $P$. Then $T$ is a 2-neighborly 2-simple 5-polytope other than a simplex. From Theorem~\ref{thm:5d}, we know that every vertex figure of $T$ at $v \in \vertices T$ is a 2-simple 2-simplicial 4-polytope $Q = Q(v)$ with $f_3(Q) = f_0(Q) = f_0(T) - 1 \ge 13$. Every facet $q$ of $Q$ is a simplicial 3-polytope. Hence, $f_2(q) = 2 f_0(q) - 4$ and
\[
2 f_2(Q) = \sum_{q \in F_3(Q)} f_2(q) = \sum_{q \in F_3(Q)} (2 f_0(q) - 4) 
= 2 f_3(Q) (y - 2),
\]
where 
\[
y = y(v) = \frac{1}{f_3(Q)} \sum_{q \in F_3(Q)} f_0(q)
\]
is the average number of vertices in a facet of $Q = Q(v)$.
On the other hand~\cite[Theorem~2(prop.~3)]{Bayer:1987},
\[
f_2(Q) \le f_3(Q) (f_3(Q) + 3) / 4.
\]
Therefore,
\begin{equation*}
%\label{eq:y(v)}
y(v) \le \frac{f_3(Q) + 3}{4} + 2 = \frac{f_0(Q) + 3}{4} + 2 = \frac{f_0(T) + 2}{4} + 2.
\end{equation*}

Let $\hat{x}$ be the average number of vertices in a facet of the 5-polytope $T$:
\[
	\hat{x} = \frac{1}{f_4(T)} \sum_{t \in F_4(T)} \hspace{-0.5em} f_0(t).
\]
and 
\[
	\hat{y} = \frac{1}{f_0(T)} \sum_{v \in F_0(T)} y(v) \le \frac{f_0(T) + 2}{4} + 2.
\]
Now we are going to show that 
\[
\hat{x} \le \hat{y} + 1
\]
and then use inequality~\eqref{eq:f_4(P)}.

Let $F_4(v)$ be the set of facets of $T$ incident to the vertex $v$:
\[
F_4(v) = \Set*{t \in F_4(T) \given v \in t}.
\]
Note, that 
\[
\sum_{v \in F_0(T)} \sum_{t \in F_4(v)} f_0(t) = \sum_{v \in F_0(T)} \sum_{q \in F_3(Q(v))} (f_0(q) + 1) = \sum_{v \in F_0(T)} |F_4(v)| (y(v) + 1)
\]
and $|F_4(v)| = f_0(T) - 1$.
Let $n = f_4(T)$ and $x_i$ be the number of vertices in $i$-th facet of $T$, $i \in [n]$.
Thus,
\[
\sum_{v \in F_0(T)} \sum_{t \in F_4(v)} f_0(t) = \sum_{i = 1}^n x^2_i
\]
and
\[
\sum_{v \in F_0(T)} |F_4(v)| (y(v) + 1) = (\hat{y} + 1)\sum_{i = 1}^n x_i.
\]
Since
\[
\left(\sum_{i = 1}^n \frac{x_i}{n}\right)^2 \le \sum_{i = 1}^n \frac{x^2_i}{n},
\]
we get
\[
\sum_{i = 1}^n \frac{x_i}{n} \le \hat{y} + 1.
\]
But $\hat{x} = \sum_{i = 1}^n \frac{x_i}{n}$ and $\hat{y} \le \frac{f_0(T) + 2}{4} + 2$.
Consequently,
\[
\hat{x} \le \frac{f_0(T) + 2}{4} + 3.
\]
From inequality~\eqref{eq:f_4(P)}, we have
\[
f_4(T) \ge \frac{4 f_0 (f_0 - 1)}{f_0 + 14}.
\]
Thus, $f_4(T) \ge 26$ for $f_0 \ge 14$.
Hence, $f_5(P) \ge 27$.
\end{proof}

%%%%%%%%%%%%%%%%%%%%%%%%%%%%%%%%%%%%%%%%%%%%%%%%%%%%%%%%%%%%%%%%%%%%%%%%%
% 
% Small number of vertices
% 
%%%%%%%%%%%%%%%%%%%%%%%%%%%%%%%%%%%%%%%%%%%%%%%%%%%%%%%%%%%%%%%%%%%%%%%%%

\section{Simplicial polytopes and polytopes with small number of vertices}
\label{sec:smallvert}

The tight lower bound for the number of $j$-faces of a simplicial $k$-neighborly $d$-polytope can be found by using $g$-theorem~\cite{Billera:1981,Stanley:1980}.

Let $M_d = (m_{i, j})$ be $(\lfloor d/2\rfloor + 1) \times (d+1)$-matrix with entries
\[
 m_{i, j} = \binom{d+1-i}{d+1-j} - \binom{i}{d+1-j}, \qquad 0 \le i \le \lfloor d/2\rfloor, \quad 0 \le j \le d.
\]
The matrix $M_d$ has nonnegative entries and the left $(\lfloor d/2\rfloor + 1) \times (\lfloor d/2\rfloor + 1)$-sub\-mat\-rix 
is upper triangular with ones on the main diagonal. 
The $g$-theorem states that the $f$-vector $f = (f_{-1}, f_0, f_1, \dots, f_{d-1})$ of a simplicial $d$-polytope 
is equal to $g M_d$ where $g$-vector $g = (g_0, g_1, \dots, g_{\lfloor d/2\rfloor})$ is an $M$-sequence
(see, e.g., \cite[Sec.~8.6]{Ziegler:1995}).

\begin{theorem}
If $P$ is a simplicial $k$-neighborly $d$-polytope ($k \le d/2$) with $n$ vertices and $n \ge d+2$, then
\[
 f_{j}(P) \ge \sum_{i=0}^{k} \left( \binom{d+1-i}{d-j} - \binom{i}{d-j}\right) \binom{n-d-2+i}{i}.
\]
In particular,
\[
 f_{d-1}(P) \ge \sum_{i=0}^{k} (d+1-2i) \binom{n-d-2+i}{i}.
\]
\end{theorem}
\begin{proof}
From the equality $f = g M_d$, we can evaluate the first $k+1$ entries of $g$:
\begin{equation}
\label{eq:GM}
	\begin{array}{r@{\ }l@{\ }l}
		g_0 &= f_{-1}(P) &= 1, \\
		g_0 m_{0, 1} + g_1 &= f_{0}(P) &= n, \\
		g_0 m_{0, 2} + g_1 m_{1, 2} + g_2 &= f_{1}(P) &= \binom{n}{2}, \\
		\hdots\\
		g_0 m_{0, k} + g_1 m_{1, k} + \dots + g_k &= f_{k-1}(P) &= \binom{n}{k}.
	\end{array}
\end{equation}

We suppose that $k \le d/2$ and $n \ge d+2$. By using induction on $j$, we prove
\begin{equation}
 \label{eq:Gi}
  g_{j} = \binom{n-d-2+j}{j}, \quad 0 \le j \le k.
\end{equation}
Obviously, $g_0 = 1$ and $g_1 = n-d-1$.
Suppose that the equality \eqref{eq:Gi} is true for $j = l$, $1 \le l \le k-1$.
From \eqref{eq:GM}, we have
\[
		\sum_{i=0}^{l} \binom{n-d-2+i}{i} \binom{d+1-i}{d-l} + g_{l+1} = \binom{n}{l+1}.
\]
Recall one of the formulation of Vandermonde's convolution:
\[
		\sum_{i} \binom{\alpha+i}{i} \binom{\beta-i}{\gamma-i} = \binom{\alpha+\beta+1}{\gamma}.
\]
Hence, $g_{l+1} = \binom{n-d-1+l}{l+1}$.

By $g$-theorem,
\[
 f_{j} = \sum_{i=0}^{\lfloor d/2 \rfloor} g_i m_{i, j+1}.
\]
Since $g$ is an $M$-sequence, we may assume $g_i = 0$ for $i > k$.
Hence,
\[
 f_{j} \ge \sum_{i=0}^{k} g_i m_{i, j-1} = \sum_{i=0}^{k} \binom{n-d-2+i}{i} m_{i, j+1}.
\]
\end{proof}

As a consequence of the theorem, we obtain~\eqref{eq:LBs2n}.
Note, that the right-hand side of \eqref{eq:LBs2n} increases monotonically with respect to both $d$ and $\Delta$.

Below we need the following lemma.

\begin{lemma}
\label{lem:3vert}
Let $P$ be a 2-neighborly $d$-polytope and $x$, $y$, $z$ be three different vertices of $P$.
If $x$, $y$, and $z$ are separated from the other vertices of $P$,
then $\conv\{x,y,z\}$ is a 2-face of $P$.
\end{lemma}
\begin{proof}
Since $x$, $y$, and $z$ are separated from the other vertices of $P$,
then \[\conv\{x,y,z\} \cap \conv(\vertices P \setminus \{x,y,z\}) = \emptyset.\]
Let $A$ be an affine hull of $\{x,y,z\}$. Thus, $\dim A = 2$ and $A \cap P = \conv\{x,y,z\}$, since every pair of vertices of $P$ form a 1-face of $P$.
Therefore, $\conv\{x,y,z\}$ is a face of $P$.
\end{proof}

%Now it is not difficult to prove the following theorem.

\begin{theorem}
	$\mn(d, v) \ge d + 7$ for $v \ge d+4$.
%Let $P$ be a $2$-neigh\-bor\-ly $d$-polytope with $f_0(P) \ge d+4$. 
%Then 
%\begin{equation}
%\label{eq:smallvert}
%f_{d-1} (P) \ge d+7.
%\end{equation}
\end{theorem}

\begin{proof}
The proof is by induction over $d$.
For $d=4$, the validity of the theorem immediately follows from~\eqref{EqNeighborly}.

Suppose that the theorem statement is true for $d = m$, $m \ge 4$.
Let $P$ be a $2$-neigh\-bor\-ly $(m+1)$-po\-ly\-tope with $f_0(P) \ge m+5$.
If $P$ is simplicial, then the statement follows from \eqref{eq:LBs2n}.

Now suppose that $P$ has a nonsimplicial facet $Q$.
Hence $Q$ is a $2$-neigh\-bor\-ly $m$-po\-ly\-tope with $f_0(Q) \ge m+2$.
By the induction hypothesis, if $f_0(Q) \ge m+4$, then $f_{m-1}(Q) \ge m+7$.
Therefore,
\[
 f_{m}(P) \ge f_{m-1}(Q) + 1 \ge m + 8.
\]
The same is true for $f_0(Q) = m+3$, since $f_{m-1}(Q) \ge m + 7$ in this case~\cite{Maksimenko:2018} (see Table~\ref{tab:1}).

Now it remains to verify the case, when $f_0(Q) = m+2$ and $P$ has no facets with $\ge m+3$ vertices. Thus, $f_{m-1}(Q) \ge m+5$ (see Table~\ref{tab:1}).
Moreover, there exist $x, y, z \in \vertices P \setminus \vertices Q$ that are separated from the other vertices of $P$. 
By Lemma~\ref{lem:3vert}, $f = \conv\{x,y,z\}$ is a~2-face of~$P$. Hence, there are at least $m-1$ facets of $P$ that are incident to $f$. Every such a facet has $\le m+2$ vertices, and three of them are $x, y, z$. Consequently, it does not have a common ridge ($(m-1)$-face) with $Q$. Therefore,
\[
f_{m}(P) \ge f_{m-1}(Q) + 1 + m-1 \ge 2m + 5.
\]
\end{proof}

Note, that by using $P_{6,10}$ (see fig.~\ref{fig:P610}) as the basis of an $m$-fold pyramid, one can deduce $\mn(d, d+4) \le d + 8$ for $d \ge 6$.
Therefore, $d + 7 \le \mn(d, d+4) \le d + 8$ for $d \ge 6$.

\section*{Acknowledgements}

The author is grateful to G\"unter Ziegler for valuable advices.

%%%%%%%%%%%%%%%%%%%%%%%%%%%%%%%%%%%%%%%%%%%%%%%%%%%%%%%%%%%%%%%%%%%%%%%%%
% 
% Bibliography
% 
%%%%%%%%%%%%%%%%%%%%%%%%%%%%%%%%%%%%%%%%%%%%%%%%%%%%%%%%%%%%%%%%%%%%%%%%%

\end{document}